\documentclass[10pt]{amsart}

\usepackage[all]{xypic}
\usepackage{rotating,graphicx}

\usepackage{latexsym}
\usepackage{amssymb}
\usepackage{amsfonts}
\usepackage{amscd}
\usepackage{amsmath,amsthm}
\usepackage{pifont}

\newtheorem{lemma}{Lemma}[section]
\newtheorem{proposition}[lemma]{Proposition}
\newtheorem{theorem}[lemma]{Theorem}
\newtheorem{corollary}[lemma]{Corollary}

\newtheorem{lem}[lemma]{Lemma}
\newtheorem{prop}[lemma]{Proposition}
\newtheorem{thm}[lemma]{Theorem}
\newtheorem{cor}[lemma]{Corollary}

{

}

\theoremstyle{definition}

\newtheorem{example}[lemma]{Example}
\newtheorem{definition}[lemma]{\sl Definition}
\newtheorem{conjecture}[lemma]{Conjecture}

\theoremstyle{remark}

\newtheorem{remark}[lemma]{Remark}

\numberwithin{equation}{section}

\newenvironment{pf}{\noindent{\bf Proof.}}{\hfill $\square$\medskip}

\def\NN{{\mathbb N}}

\def\PP{{\mathbb P}}

\def\ZZ{{\mathbb Z}}

\def\0ol{{\bar 0}}
\def\1ol{{\bar 1}}
\def\2ol{{\bar 2}}
\def\ol2{{\bar 2}}
\def\3ol{{\bar 3}}
\def\4ol{{\bar 4}}
\def\5ol{{\bar 5}}
\def\6ol{{\bar 6}}
\def\7ol{{\bar 7}}
\def\8ol{{\bar 8}}
\def\9ol{{\bar 9}}

\def\bold0{{\bf 0}}
\def\bold1{{\bf 1}}
\def\bold2{{\bf 2}} 
\def\bold3{{\bf  3}}
\def\bold4{{\bf 4}}
\def\bold5{{\bf 5}}
\def\bold6{{\bf 6}}
\def\bold7{{\bf 7}}
\def\bold8{{\bf 8}}
\def\bold9{{\bf 9}}

\def\P2Skly{\PP^2_{Skly}}

\def\Alt{\operatorname {Alt}}

\def\Ext{\operatorname {Ext}}

\def\GL{\operatorname {GL}}

\def\Hom{\operatorname {Hom}}

\def\ker{\operatorname {ker}}

\def\pd{{\operatorname {\partial}}}

\def\th{\operatorname {th}}    

\def\Aut{\operatorname{Aut}}

\def\det{\operatorname{det}}

\def\dim{\operatorname{dim}}

\def\Ext{\operatorname{Ext}}

\def\gldim{\operatorname{gldim}}

\def\Hom{\operatorname{Hom}}

\def\id{\operatorname{id}}
\def\Id{\operatorname{Id}}
\def\Im{\operatorname{Im}}

\def\pd{{\partial}}

\def\Sym{\operatorname{Sym}}

\def\ul1{\operatorname{\underline{1}}}

\def\l{\leftarrow}

\def\d{\downarrow}

\def\a{\alpha}
\def\b{\beta}
\def\c{\gamma}
\def\d{\delta}
\def\e{\epsilon}

\def\l{\lambda}

\def\s{\sigma}
\def\t{\tau}

\def\fS{{\mathfrak S}}

\def\fm{{\mathfrak m}}

\def\sM{{\sf M}}

\def\sfv{{\sf v}}
\def\sfw{{\sf w}}

\def\cal{\mathcal}

\def\cD{{\cal D}}
\def\cE{{\cal E}}

\def\dirlim{\mathop{\vtop{\baselineskip -100pt\lineskip -1pt\lineskiplimit 0pt
\setbox0\hbox{lim}\copy0\hbox to \wd0{\rightarrowfill}}}\limits}
\def\invlim{\mathop{\vtop{\baselineskip -100pt\lineskip -1pt\lineskiplimit 0pt
\setbox0\hbox{lim}\copy0\hbox to \wd0{\leftarrowfill}}}\limits}

\def\I11{{1 \kern -0.8pt \! \mbox{l}}}
\def\mumu{{\mu\kern-4.2pt\mu}}
\def\bfmu{{\mu\kern-4.2pt\mu}}
\def\2slash{\backslash \! \backslash}

\def\boxtimes{\setbox0\hbox{$\Box$}\copy0\kern-\wd0\hbox{$\times$}}

\def\hdot{{\:\raisebox{2pt}{\text{\circle*{1.5}}}}}

\def\hdet{{\rm hdet}} 

\def\Aut{\operatorname {\rm Aut}}

\def\<{\langle}
\def\>{\rangle}

\begin{document}

 
\title[$m$-Koszul AS-regular Algebras]{ $m$-Koszul Artin-Schelter regular Algebras}

\author{Izuru Mori and S. Paul Smith}

\address{Department of Mathematics, Graduate School of Science, Shizuoka University, Shizuoka 422-8529, Japan.}
\email{simouri@ipc.shizuoka.ac.jp}
\address{Department of Mathematics, Box 354350, University of Washington, Seattle, WA 98195,USA.}
\email{smith@math.washington.edu}



\thanks{The first author was supported by Grant-in-Aid for Scientific Research (C) 91540020. The first author thanks the University of Washington for its hospitality during the period that this work was done.
}

\begin{abstract} 
This paper studies the homological determinants and Nakayama automorphisms of not-necessarily-noetherian 
$m$-Koszul twisted Calabi-Yau or, equivalently,  
$m$-Koszul Artin-Schelter  regular, algebras.  Dubois-Violette showed that such an algebra is 
isomorphic to a derivation quotient algebra $\cD(\sfw,i)$ for a unique-up-to-scalar-multiples twisted superpotential $\sfw$. By
definition, $\cD(\sfw,i)$ is the quotient of the tensor algebra $TV$, where $V=\cD(\sfw,i)_1$, by $(\pd^i\sfw)$, the ideal generated by all $i^{\th}$-order left partial derivatives of $\sfw$. The restriction map $\s \mapsto \s|_V$ is used to identify the group of graded algebra automorphisms of 
$\cD(\sfw,i)$ with a subgroup of $\GL(V)$. 
We show that the homological determinant of a graded algebra automorphism $\s$ of an $m$-Koszul Artin-Schelter regular algebra
$\cD(\sfw,i)$ is given by the formula $\hdet(\s)\sfw =\s^{\otimes (m+i)}(\sfw)$. 
It follows from this that the homological determinant of the Nakayama automorphism of an $m$-Koszul Artin-Schelter regular algebra is 1.  
As an application, we prove that the homological determinant and the usual determinant coincide for most quadratic noetherian Artin-Schelter regular algebras of dimension 3.  
\end{abstract}

\maketitle

\section{Introduction}

We fix a field $k$. All vector spaces will be $k$-vector spaces. All algebras will be $\NN$-graded $k$-algebras.
Such an algebra, $A=A_0 \oplus A_1 \oplus \cdots$ is {\sf connected} if $A_0=k$. In that case, $k=A/A_{\ge 1}$ is a
graded left  $A$-module concentrated in degree 0. 

\subsection{}
Two classes of graded algebras play a central role in non-commutative projective algebraic geometry, namely twisted 
Calabi-Yau algebras and Artin-Schelter regular algebras. The definitions of these algebras vary a little from one paper to 
another but in this paper they will be defined in such a way that they are, in fact, the same (see \S\ref{ssect.CY}). 
The algebras we study need not be noetherian.

This paper concerns the homological determinants and Nakayama automorphisms of $m$-Koszul twisted Calabi-Yau or, equivalently,  
$m$-Koszul Artin-Schelter  regular, algebras. 

\subsubsection{}
A connected graded algebra $A$ is {\sf $m$-Koszul} if it is finitely presented and 
its relations are homogeneous of degree $m$ and $\Ext^i_A(k,k)$ is concentrated in
a single degree for all $i$. $2$-Koszul algebras are Koszul algebras in the ``classical'' sense. $m$-Koszul algebras 
were introduced by R. Berger \cite{B1}.

\subsubsection{}
A connected graded algebra $A$ is  {\sf Artin-Schelter regular} (AS-regular, for short) of dimension $d$ if  $\gldim(A)=d<\infty$, and
\begin{equation}
\label{AS-Gor-cond}
\Ext^i_A(k,A) \cong \begin{cases} 
					k(\ell) & \text{if $i=d$}
					\\
					0 & \text{if $i \ne d$}
				\end{cases}
\end{equation}
for some $\ell \in \ZZ$. Commutative Artin-Schelter regular algebras are polynomial rings. 
The number $\ell$ in (\ref{AS-Gor-cond}) is called the  {\sf Gorenstein parameter} of $A$.

\subsection{Notation}
\label{ssect.notn1}
The following notation will be used throughout the paper.  
\begin{enumerate}
\item{}
$A$ denotes an arbitrary connected $\NN$-graded $k$-algebra.
  \item
  $S$ denotes an $m$-Koszul AS-regular algebra of dimension $d$ and Gorenstein parameter $\ell$.
  \item 
  $V=S_1$ and $R=\ker\big(V^{\otimes m} \xymatrix{\ar[rr]^{\rm multiplication}&& } S_m\big)$; thus $S\cong TV/(R)$.
  \item{}
  $\Aut(S)$ is the group of graded $k$-algebra automorphisms of $S$. Since an automorphism of $S$ is determined by its restriction to $S_1$
  we always view $\Aut(S)$ as a subgroup of $\GL(V)$, $\Aut(S)=\{\s \in \GL(V) \; | \; \s^{\otimes m}(R)=R\}$.
  \item{}
  $\hdet:\Aut(S) \to k^\times$ is the homological determinant (see Theorem \ref{thm.JZ}).
  \item
  $\e\in \Aut(S)$ is multiplication by $(-1)^i$ on $S_i$. 
  \item
  $\sfw$, or $\sfw_S$ when we need to emphasize $S$, 
   is a basis for the subspace
   \newline 
   $\bigcap_{s+t=\ell-m} V^{\otimes s} \otimes R \otimes V^{\otimes t}$ of  $V^{\otimes \ell}$. 
  \item 
  $E=\Ext^*_S(k,k)$ and $E^i=\Ext^i_S(k,k)$.
  \item{}
 $\nu$ is the Nakayama automorphism of $S$ (defined in \S\ref{ssect.CY}).
   \item{}
 $\mu$ is the Nakayama automorphism of  $E$ (defined in \S\S\ref{ssect.nak.autom1.5} and \ref{ssect.nak.autom2}).
\end{enumerate}
Under the convention in (4), $\nu \in \GL(V)$. However, if $m \ge 3$, $E$ is not generated by $V^*=\Ext^1_S(k,k)$ so it
is not natural to identify $\Aut(E)$ with a subgroup of $\GL(V^*)$. Nevertheless, we can, and do, consider $\mu|_{V^*}$. 

Proposition \ref{prop.w} explains why the intersection in (7) is 1-dimensional. 

\subsection{Results I}

In this section, we adopt the notation in \S\ref{ssect.notn1}.

\begin{theorem} \textnormal {(Theorem \ref{thm.ws})} 
$\Aut (S)=\{\s \in \GL(V) \; | \; \s^{\otimes \ell}(k\sfw)=k \sfw\}$.
\end{theorem}

\begin{theorem} \textnormal{(Theorem \ref{thm.msss})} 
If $\s\in \Aut(S)$, then $\s^{\otimes \ell}(\sfw)=\hdet (\s)\sfw$.
\end{theorem}  
 
The following theorem makes \cite [Thms. 6.2 and 6.8]{BSW} less mysterious (cf. \cite[Thm. 4.12]{WZ}).  

\begin{theorem} \textnormal{(Corollary \ref{cor.sp})} 
$S$ is Calabi-Yau if and only if $\sfw$ is a $(-1)^{d+1}$-twisted superpotential in the sense of Definition \ref{de.tw.spp}.
\end{theorem}

Conditions (2) and (3) in the next result are closed conditions so (1) says that the homological determinant equals the determinant
for ``almost all'' noetherian quadratic AS-regular algebras of dimension 3 .

\begin{theorem} \textnormal {(Theorem \ref{thm.cwhd})}
Suppose $d=3$ and $m=2$, i.e., $S$ is a 3-dimensional Artin-Schelter regular algebra on 3 generators.  
The following are equivalent: 
\begin{enumerate}
\item{} There exists $\s\in \Aut(S)$ such that $\hdet (\s)\neq \det (\s)$. 
\item{} $c(\sfw)\in \Sym^3V$. 
\item{} $R\subseteq \Sym^2V$. 
\end{enumerate}
\end{theorem}

\subsection{Results II} 
We continue to use the notation in \S\ref{ssect.notn1}.

\begin{theorem} 
\label{thm.BM.nak-intro} 
\textnormal{(Thm. 4.8)}
The Nakayama automorphism of $\Ext^*_S(k,k)$ is  $(\e^{d+1}\nu )^!$.\footnote{See \S\S \ref{sse.m.Koszul} and \ref{ssect.automs} for the definition of $(\cdot )^!$.}
\end{theorem} 

Lu, Mao, and Zhang, have proved a general result showing that the Nakayama automorphisms of
suitable algebras $A$ belong to the center of the full automorphism group   \cite[Theorem 4.2]{LMZ}. 
Using Theorem \ref{thm.BM.nak-intro}, we give a different proof of the fact  that the Nakayama automorphism of $S$ is in the center of $\Aut(S)$ (Corollary \ref{cor.mnc}).

The next two results are non-noetherian $m$-Koszul versions of   \cite [Thm.5.3]{RRZ2} 
and \cite[Thm.5.4(a)]{RRZ}. Our proofs are different and are, perhaps, simpler.  

\begin{theorem} \textnormal{(Theorem \ref{thm.hdn})}
$\hdet (\nu)=1$. 
\end{theorem}

\begin{theorem} \textnormal{(Theorem \ref{thm.msn})} 
Let $\s\in \Aut(S)$. The Zhang twist, $S^{\s}$, is also an $m$-Koszul AS-regular algebra of Gorenstein parameter $\ell$, and 
its Nakayama automorphism is $\hdet(\s)^{-1}\s^{\ell}\nu$.  
\end{theorem}

The following theorem extends \cite [Cor. 9.3]{V} and \cite [Prop. 6.5]{BM} to the higher dimensional case. 

\begin{theorem} \textnormal{(Theorem \ref{thm.mnw2})} 
Let $x_1, \dots, x_n$ be  a basis  for $V=S_1$. Let  ${\bf x}^t=(x_1, \dots, x_n)$.
There is a unique  $n\times n$ matrix  $\sM$ with entries in $V^{\otimes \ell-2}$ such that  $\sfw={\bf x}^t\sM{\bf x}$
 and a unique  $Q\in \GL_n(k)$ such that $({\bf x}^t\sM)^t=Q\sM{\bf x}$.
The Nakayama automorphism of $S$ is the unique $\nu \in \GL(V)$ such that $\nu =(-1)^{d+1}Q^{-t}\in \GL(n, k)$ with respect to this basis.  In particular, $S$ is Calabi-Yau if and only if $Q=(-1)^{d+1}\Id$.  
\end{theorem} 

\subsection{Methods}
We continue to use the notation in \S\ref{ssect.notn1}.

Almost all our results depend on Theorem \ref{thm.DV}  below which says that $S$ is completely determined by the 1-dimensional
subspace $k\sfw$ defined in  \S\ref{ssect.notn1}(7). 

\subsubsection{Derivation-quotient algebras}
Let $V$ be a vector space and $W$ a subspace of some tensor power,  $V^{\otimes p}$ say.
We introduce the following notation:
\begin{align*}
\pd W  & \; :=\; \big\{(\psi \otimes \id^{\otimes p-1})(\sfw) \; \big\vert \; \psi \in V^*, \, \sfw \in W\big\},
\\
\pd^{i+1}W  &\; :=\;  \pd(\pd^i W) \quad \hbox{for all $i \ge 0$,}
\\
\cD(W, i)  &\; :=\; TV/(\partial^iW).
\end{align*}
The space $\pd^iW$ appears in \cite[\S4]{DV} where it is denoted $W^{(p-i)}$. 

Let $\sfw\in V^{\otimes \ell}$. 
Following \cite{BSW}, we introduce the notation
$$
\cD(\sfw, i) \; :=\;  \cD(k\sfw, i)
$$
and call  $\cD(\sfw, i)$ the {\sf $i$-th order derivation-quotient algebra} of $\sfw$.

\begin{theorem}
[Dubois-Violette]
 \textnormal{\cite[Thm. 11]{DV}}
\label{thm.DV} 
In the notation of  \S\ref{ssect.notn1} 
$$
S\cong\cD(\sfw, \ell-m).
$$
\end{theorem} 

All results in this paper concern automorphisms of $m$-Koszul AS-regular algebras so the following simple observation
plays a central role.   
 
 \begin{thm}
 \textnormal{(Theorem \ref{thm.ws})}
$\Aut\big(\cD(\sfw,i)\big)=\{\s \in \GL(V) \; | \; \s^{\otimes \ell}(k\sfw)=k\sfw\}$.
\end{thm}

\section{Preliminaries}

\subsection{Calabi-Yau algebras}
\label{ssect.CY}
Let  $A$ be a graded $k$-algebra, $A^\circ$ its opposite algebra, and $A^e=A \otimes A^\circ$.
We consider $A$ as a left $A^e$-module via $(a \otimes b)\cdot c=acb$. If $\nu$ is an automorphism of $A$ we denote by ${{}_{{}_\nu}} A_{{}_1}$ the left $A^e$-module that is $A$ as a graded vector space with action $(a \otimes b)\cdot c=\nu(a)cb$. 
 We say $A$ is  {\sf twisted Calabi-Yau} of {\sf dimension} $d$ if it has a finite-length resolution as a left $A^e$-module by finitely 
 generated projective $A^e$-modules and there is an isomorphism
\begin{equation}
\label{CY-cond}
\Ext^i_{A^e}(A,A^e) \cong \begin{cases} 
					{{}_{{}_\nu}} A_{{}_1}(\ell) & \text{if $i=d$}
					\\
					0 & \text{if $i \ne d$}
				\end{cases}
\end{equation}
of graded right $A^e$-modules for some integer $\ell$ and some graded $k$-algebra automorphism $\nu$. We call $\nu$ the {\sf Nakayama automorphism} of $A$. (Some authors call $\nu^{-1}$ the Nakayama automorphism; see, e.g., \cite{RRZ}.) We say $A$ is  {\sf Calabi-Yau} if 
it is twisted Calabi-Yau and $\nu=\id_A$. If $A$ is  twisted Calabi-Yau of dimension $d$, then it has global dimension $d$. 

When $A$ is Calabi-Yau of dimension $d$ we often say it is $d$-Calabi-Yau or $d$-CY.

\begin{prop}
\cite[Prop.3.1(1)]{SZ}
If $A$ is Artin-Schelter regular, then ${}_Ak$ has a finite free resolution, i.e., a finite projective
resolution in which every term is a finitely generated free left $A$-module. In particular, $A$ is finitely presented. 
\end{prop}

\begin{cor}
If $A$ is Artin-Schelter regular, then there is a finite dimensional graded vector space $V$ and a finite dimensional graded subspace $R$
of the tensor algebra $TV$ such that  $A=TV/(R)$.
\end{cor}

\begin{lem}
\cite[Lem.1.2]{RRZ}
\label{lem.RRZ}
A connected graded algebra  is twisted Calabi-Yau  of dimension $d$ if and only if it is Artin-Schelter regular algebra of dimension $d$. 
In that case, the numbers $\ell$ in (\ref{CY-cond}) and (\ref{AS-Gor-cond}) are the same. 
\end{lem}

\subsection{The homological determinant}
In \cite{JZ}, Jorgensen and Zhang introduced the homological determinant in order to prove for AS-Gorenstein rings, which are defined in
Theorem \ref{thm.JZ}, non-commutative analogues of several important results about commutative connected graded Gorenstein algebras.  
Since then the homological determinant has played an important role in the invariant theory of non-commutative rings. 

The homological determinant of an AS-Gorenstein algebra, $A$, is a group homomorphism $\hdet:\Aut(A) \to k^\times$.
If $SV$ is a symmetric algebra, i.e., a polynomial ring, on a finite dimensional vector space $V$, then $\Aut(SV)=\GL(V)$ and the homological determinant is the same as usual determinant.

It is often difficult to compute the homological determinant.
  
The {\sf homological determinant} is the function $\hdet$ in the next theorem.

\begin{thm}
\cite[Lem. 2.2, Prop. 2.5]{JZ}
\label{thm.JZ}
Let $A$ be a connected graded $k$-algebra. 
Suppose there is an integer $d$ such that left and right injective dimensions of $A$ as a graded $A$-module equal $d$ and that  
$$
 \Ext_A^i(k,A) = \Ext_{A^\circ}^i(k,A) \cong \d_{id} \, k(\ell).
$$ 
for some integer $\ell$. (Such an $A$ is said to be AS-Gorenstein.)
\begin{enumerate}
  \item 
  Each $\s \in \Aut(A)$ induces an isomorphism $H^d_\fm(\s):H^d_\fm(A) \to H^d_\fm(A)$ of graded vector spaces where $H^d_\fm(A)$ is the 
  $d^{\th}$ local cohomology group (see \cite{JZ} for the definition.)
  \item 
 There is an isomorphism $\psi:H^d_\fm(A)  \to A'(\ell)$ of graded vector spaces where $A'$ is the Matlis dual, 
 $A'=\oplus_{i=0}^\infty \Hom_k(A_i,k)$.
 \item 
There is a group homomorphism
$$
\hdet: \Aut(A) \to k^\times, \qquad \s \mapsto \hdet(\s),
$$
such that the diagram
$$
\xymatrix{
H^d_\fm(A) \ar[d]_\psi \ar[rr]^{H^d_\fm(\s)} && H^d_\fm(A)  \ar[d]^\psi 
\\
A'(\ell) \ar[rr]_{\hdet(\s)\,(\s^{-1})'} && A'(\ell)
}
$$
commutes. 
\item{}
The maps $H^d_\fm(\s)$ and $(\s^{-1})'$ are $\s$-linear: a degree-preserving linear map $f:M \to N$ between two graded left $A$-modules is 
$\s$-linear if $f(am)=\s(a)f(m)$ for all $a \in A$ and $m \in M$.
\end{enumerate}
\end{thm}

\subsection{}

Let  $\phi:V^{\otimes \ell}\to V^{\otimes \ell}$ be the linear map
\begin{equation}
\label{defn.phi}
\phi(v_1\otimes v_2\otimes \cdots \otimes v_{\ell-1}\otimes v_{\ell})  \; := \; v_{\ell}\otimes v_1\otimes \cdots \otimes v_{\ell-2}\otimes v_{\ell-1}.
\end{equation}

\begin{definition} 
\label{de.tw.spp}
Let $\sfw\in V^{\otimes \ell}$ and $\s \in \GL(V)$. We call $\sfw$ a
\begin{enumerate}
  \item 
  {\sf superpotential} if $\phi (\sfw)=\sfw$; 
  \item 
  {\sf $\s$-twisted superpotential} if $(\s\otimes \id^{\otimes \ell-1})\phi (\sfw)=\sfw$; 
  \item 
  a {\sf twisted superpotential} if it is a  $\s$-twisted superpotential for some $\s$. 
\end{enumerate}  
\end{definition}

\begin{remark} 
Our terminology does not agree with that in \cite{BSW}: there, $\sfw \in V^{\otimes \ell}$ is called a twisted superpotential if it is invariant under the 
map $v_1\otimes \ldots \otimes v_\ell \mapsto (-1)^{\ell+1}\s(v_\ell)\otimes v_1\otimes  \ldots \otimes v_{\ell-1}$ for some $\s \in \GL(V)$.
\end{remark}

\subsection{}
\label{VV*.convention}
For a linear transformation $\theta:U\to V$, we define $\theta^*:V^*\to U^*$ by $\theta^*(\xi)=\xi\circ \theta$.  
In order to identify $(V^*)^{\otimes i}$ with $(V^{\otimes i})^*$,  we adopt the convention that 
 $$
 (\xi_1\otimes \ldots \otimes \xi_i)(v_1\otimes \ldots \otimes v_i) \; := \; \xi _1(v_1)\ldots \xi_i(v_i).
 $$

\subsection{The  $m$-Koszul complex}
\label{sse.m.Koszul}

The {\sf $m$-homogeneous dual} of an $m$-homogeneous algebra $A=TV/(R)$   is the algebra
$$
A^! \; := \; \frac{TV^*}{(R^\perp)}
$$
where $R^\perp \subseteq \big(V^{\otimes m}\big)^*=(V^*)^{\otimes m}$ consists of the functions vanishing on $R$.

Since $\Aut( A)  = \{ \s \in \GL(V) \; | \; \s^{\otimes m}(R)=R\}$, the anti-isomorphism  $\GL(V) \to \GL(V^*)$, 
$\s \mapsto \s^*$, restricts to an anti-isomorphism  
$$
 \Aut( A) \to  \Aut( A^!)
 $$
that we denote by $\s \mapsto \s^!$.

\begin{lem}
\label{lem.Wi}
Let $A=TV/(R)$ be an $m$-homogeneous algebra. 
For all $i \ge 0$, define 
\begin{equation}
\label{defn.Wi}
W_i\; := \;
\begin{cases} 
V^{\otimes i} & \textnormal { if } 0\leq i \leq m-1,
 \\
\cap _{s+m+t=i}V^{\otimes s}\otimes R\otimes V^{\otimes t} \; \subseteq  \; V^{\otimes i} &\textnormal { if } i \ge m.
\end{cases}
\end{equation} 
The quotient map $\Phi:TV^* \to A^!$ induces isomorphisms
$A^!_i \; \cong \; W_i^*$ for all $i \ge 0$.
\end{lem}
\begin{pf}
Let $\Phi_i$ denote the restriction of $\Phi$ to $(V^*)^{\otimes i}= ( V^{\otimes i})^*$. 
There is an exact sequence $0 \to W_i^{\perp}\to (V^*)^{\otimes i}\to W_i^*\to 0$ so
it suffices to show that $\ker(\Phi_i)=W_i^\perp$.

Suppose $0\leq i\leq m-1$. Then $A^!_i=(V^*)^{\otimes i}$ so $\ker(\Phi_i)=0$.
But $W_i=V^{\otimes i}$ so $W_i^{\perp}=0$ also. Hence  $\ker(\Phi_i)=W_i^{\perp}$ when $i\le m-1$.

Suppose $i\geq m$. Then
\begin{align*}
W_i^{\perp} & \; = \; \left(\bigcap _{s+m+t=i}V^{\otimes s}\otimes R\otimes V^{\otimes t}\right)^{\perp} \\
&  \; = \;  \sum _{s+m+t=i}(V^{\otimes s}\otimes R\otimes V^{\otimes t})^{\perp} \\
&  \; = \;  \sum _{s+m+t=i}(V^*)^{\otimes s}\otimes R^{\perp}\otimes (V^*)^{\otimes t} \\
&  \; = \;  \ker(\Phi_i).
\end{align*}
 The proof is complete. 
\end{pf}

The space $W_i$ is denoted  by $W^{(i)}$ in \cite{DV} and by $J_i$ in \cite{B1}.

Let $\{x_i\}$ be a basis for $V$ and $\{\xi_i\}$ the dual basis for $V^*=A_1^!$. The element $e:=\sum_i x_i \otimes \xi_i$ belongs to
$A \otimes A^!$ and is independent of the choice of basis for $V$. Right multiplication by $e$ is a left $A$-module homomorphism
$$
\xymatrix{
A \otimes (A_n^!)^* \ar[r]^>>>>>>{\cdot e} & A \otimes (A_{n-1}^!)^*. 
}
$$
When $A$ is an $m$-homogeneous  algebra $e^m=0$. 

\begin{thm}
[Berger] \textnormal{\cite [Thm. 2.4]{BM}}
\label{thm.berger}
Let $A$ be an $m$-homogeneous algebra and define $P_i=A \otimes (A^!_i)^*$ for $i \ge 0$. The complex 
$$
\xymatrix{
\cdots 
 \ar[r]^{\cdot e} &
 P_{2m} \ar[r]^{\cdot e^{m-1}}  & P_{m+1}\ar[r]^{\cdot e} &   P_m \ar[r]^{\cdot e^{m-1}} & P_1 \ar[r]^{\cdot e} & P_0 \ar[r] & k \ar[r] &0
}
$$
is exact  if and only if $A$ is $m$-Koszul.  
\end{thm}

If $A$ is an $m$-homogeoneous algebra we call the complex in Theorem \ref{thm.berger} the {\sf $m$-Koszul complex} for $A$.

\subsection{The element $\sfw$ in  \S\ref{ssect.notn1}(7).}
The definition of an AS-regular algebra $A$ implies that if $P_\hdot$ is a deleted projective resolution of ${}_Ak$, then $\Hom_A(P_\hdot,A)$ is  a deleted projective resolution of $k_A$.  This implies the following well-known result.

\begin{prop}
With the notation in \S\ref{ssect.notn1}, let $P_i=S \otimes (S^!_i)^*$ for $i \ge 0$. The left-hand part of the 
minimal resolution of ${}_Sk$ looks like
$$
 \xymatrix{
 0 \ar[r] & P_{\ell} \ar[r]^{\cdot e} &
 P_{\ell-1} \ar[r]^{\cdot e^{m-1}}  & P_{\ell-m}\ar[r]^{\cdot e} & \cdots 
}
$$
where   $\dim_k(S^!_\ell)=1$, $\dim_k(S^!_{\ell-1})=\dim_k(V)$, and $\dim_k(S^!_{\ell-m})=\dim_k(R)$. 
\end{prop}

\subsubsection{Remark}
With the notation in \S\ref{ssect.notn1}, the fact that the arrow $P_{\ell} \to P_{\ell-1}$ is multiplication by $e$ implies that $d$ is odd if 
$m \ge 3$. 

\subsubsection{}
The intersection in the next proposition is $W_\ell$ which is isomorphic to $(S^!_\ell)^*$.
 
 \begin{prop}
 \label{prop.w}
 With the notation in \S\ref{ssect.notn1}, 
 $\bigcap_{s+t=\ell-m} V^{\otimes s} \otimes R \otimes V^{\otimes t}$ is a 1-dimensional subspace of  $V^{\otimes \ell}$.
 \end{prop}

 We retain the notation in \S\ref{ssect.notn1}. Dubois-Violette's theorem (Theorem \ref{thm.DV} above) says that $S$ is isomorphic to 
 $\cD(\sfw',\ell-m)$ for {\it some} twisted superpotential $\sfw' \in V^{\otimes \ell}$. Proposition \ref{prop.sdw} below shows 
 that we can take $\sfw'$ to be the element $\sfw$ defined in  \S\ref{ssect.notn1}(7).  
 This result is implicit in \cite[Thm. 11]{DV} however the following proof is short and explicit (cf. \cite[Thm.4.12]{WZ}).

\begin{lemma} 
\label{lem.spwl} 
Suppose $m,p,q$ are integers $\ge 0$. Let $R$ be a subspace of $V^{\otimes m}$ and 
$\sfw\in V^{\otimes p}\otimes R\otimes V^{\otimes q}$. 
If $\sfw$ is a twisted superpotential, then 
$$
\sfw \; \in \; \bigcap_{s+t=p+q} V^{\otimes s}\otimes R\otimes V^{\otimes t},
$$
where the intersection is taken over all integers $s,t \ge 0$ such that $s+t=p+q$.
\end{lemma} 
\begin{proof} 
Let $\s \in \GL(V)$ be such that $(\s \otimes \id)\phi(\sfw)=\sfw$. Then 
$$\begin{array}{ll}
\sfw=(\s\otimes \id)\phi (\sfw)\in V^{\otimes p+1}\otimes R\otimes V^{\otimes q-1} & \textnormal { if } q\geq 1, \text{ and}\\
\sfw=\phi^{-1}(\s^{-1}\otimes \id)\sfw\in V^{\otimes p-1}\otimes R\otimes V^{\otimes q+1} & \textnormal { if } p\geq 1. \\
\end{array}$$
An induction argument completes the proof.  
\end{proof}

 \begin{proposition} 
 \label{prop.sdw} 
 Let $\ell \ge m \ge 2$. Let $\sfw \in V^{\otimes \ell}$. 
 Suppose $\cD(\sfw, \ell-m)$ is an $m$-Koszul AS-regular algebra of Gorenstein parameter $\ell$. Let $R=\pd^{\ell-m}(k\sfw)$. 
 If $\sfw$ is a twisted superpotential, then 
 $$
 k\sfw \; = \;  \bigcap_{s+t+m=\ell} V^{\otimes s}\otimes R\otimes V^{\otimes t}, 
 $$
 where the intersection is taken over all integers $s,t \ge 0$ such that $s+m+t=\ell$.
\end{proposition} 
\begin{proof}  
By definition, $\cD(\sfw,\ell-m)=TV/(R)$. It follows from the definition of $\pd^{\ell-m}$ that $\sfw \in V^{\otimes \ell-m}\otimes R$.  
Since $\sfw$ is a twisted superpotential, Lemma \ref{lem.spwl} tells us that $\sfw\in W_{\ell}$. But $\dim_k(W_\ell)=1$.  
\end{proof}

\subsection{The algebra structure on $\Ext^*_S(k,k)$}
\label{sse.cE}

The {\sf Yoneda Ext-algebra} of  a connected graded $k$-algebra $A$  is 
$$
E(A) \; := \; \bigoplus_{i=0}^\infty\Ext^i_A(k, k).
$$
Because $A$ is a graded $k$-algebra, $E(A)$ is bi-graded. Several notations are used for its homogeneous components, namely $E^{ij}(A)=E^i(A)_{-j}=\Ext^i_A(k, k)_{-j}$. 

We denote the product on $A^!$ by $f\cdot g$. 

Following \cite[\S2]{HL}, we define
$$
\rho (i):=\begin{cases} jm & \textnormal { if $i=2j$ } \\  jm+1 & \textnormal { if $i=2j+1$}\end{cases}
$$
and the bigraded vector space 
$$
\cE \; =\; \cE(A) = \bigoplus \cE^i_j \qquad \hbox{where} \qquad
\cE^i_j = \begin{cases} 
             0 & \text {if $j \ne -\rho(i)$}
             \\
             A^!_{\rho(i)} &  \text {if $j = -\rho(i)$.}
            \end{cases} 
$$
We define a multiplication $*$ on $\cE$ as follows:  if $f \in \cE^i$ and $g \in \cE^j$, then
\begin{equation}
\label{mult.on.cE}
f * g \; =\; \begin{cases} 
             (-1)^{ij}f\cdot g & \text{if $m=2$}
             \\
             f \cdot g   & \text{if $m>2$ and at least one of $i$ and $j$ is even}
             \\
             0 & \text{otherwise.}
             \end{cases}
\end{equation}

\begin{prop}
\textnormal{\cite[Prop. 3.1]{BM}, \cite[Prop. 2.3]{HL}.}
Let $A$ be an $m$-Koszul algebra. The direct sum of the vector space isomorphisms $A^!_{\rho(i)} \to \Ext^i_A(k,k)$ is an 
isomorphism of bigraded $k$-algebras, 
$$
\cE^*_\hdot  \stackrel{\sim}{\longrightarrow} \Ext^*_A(k,k)_\hdot.
$$
 \end{prop}
 
When $A$ is an $m$-Koszul algebra we identify $\cE$ with $\Ext^*_A(k,k)$.

\subsubsection{The automorphisms $\s^!$ of $\cE$}
\label{ssect.automs}
It follows from the definition of multiplication on $\cE$ that every graded $k$-algebra automorphism of $A^!$ ``restricts'' to 
a graded $k$-algebra automorphism of $\cE$. Thus, if $\s \in \Aut(A)$ and $\s^!$ is the automorphism of $A^!$ induced by $\s^*$, then
$\s^!$ ``restricts'' to an automorphism of $\cE$. We continue to denote that automorphism of $\cE$ by $\s^!$
(cf., \cite [p.57]{WZ}).    

Let $\s \in \Aut(A)$. By Lemma \ref{lem.Wi}, $\Ext^i_A(k,k)$ is a quotient of $(V^*)^{\otimes i}$ so the action of $\s^!$
on $\Ext^i_A(k,k)$ is induced by the action of $(\s^!)^{\otimes i}$ on $(V^*)^{\otimes i}$. In particular, that is the meaning of $\s^!$ in 
(\ref{Wu.Zhu.hdet}).\footnote{The convention we use in \S\ref{VV*.convention} to identify $(V^*)^{\otimes i}$ with 
$(V^{\otimes i})^*$ differs from that in \cite{WZ}. The paper \cite{WZ} uses the convention 
$(\xi_1\otimes \cdots \otimes \xi_i)(v_1\otimes \cdots \otimes v_i)=\xi _i(v_1)\cdots \xi_1(v_i)$.
Under that convention, if $\s_1, \dots, \s_i\in \GL(V)$ and $\xi_1, \dots, \xi _i\in V^*$, 
then $(\s_1\otimes \cdots \otimes \s_i)^*(\xi_1\otimes \cdots \otimes \xi_i)=\s_i^*(\xi_1)  \cdots \s_1^*(\xi_i)$.
However, if  $\s\in \GL(V)$ both conventions give $(\s\otimes \cdots \otimes \s)^*(\xi _1\otimes \cdots \otimes \xi _i)=\s^*(\xi _1) \cdots  \s^*(\xi _i)$. Therefore Proposition \ref{prop.WZ} holds as stated under both conventions.}

\begin{prop} \textnormal {\cite [Prop. 1.11]{WZ}} \label{prop.WZ}  
 With the notation in \S\ref{ssect.notn1}, the homological determinant of $\s\in \Aut( S)$ has the property that 
\begin{equation}
\label{Wu.Zhu.hdet}
 \hdet (\s) u = \s^!(u)
 \end{equation}
 for all $u\in \Ext^d_S(k,k)$.
 \end{prop}

\section{Homological Determinants} 
\label{ssect.hom.det}

Theorem \ref{thm.msss}  gives a simple formula for the homological determinant  of an $m$-Koszul AS-regular algebra. 
This formula is the key to several results in later sections.

 \begin{lem}
Let $\s \in \GL(V)$ and $W \subseteq V^{\otimes j}$. If $\s^{\otimes j}(W)=W$, then  $\s \in \Aut\big(\cD(W,i)\big)$ for all $i$.
\end{lem}
\begin{pf}
To prove the result it suffices to show that $\s^{\otimes j-i}(\partial ^i W)=\partial ^iW$ for all $i=0,\ldots,j$.
First we prove the claim for $i=1$. It is clear that 
\begin{align*}
\s^{\otimes j-1}(\partial W) 
& =\{\s^{\otimes j-1}(\psi \otimes \id^{\otimes j-1})(\sfv)\; | \; \psi\in V^*, \sfv\in W\} \\
& =\{(\psi \otimes \s^{\otimes j-1})(\sfv)\; | \; \psi \in V^*, \sfv\in W\} \\
& =\{(\psi\s \otimes \s^{\otimes j-1})(\sfv)\; | \; \psi \in V^*, \sfv\in W\} \\
& =\{(\psi\otimes \id^{\otimes j-1})(\s \otimes \s^{\otimes j-1})(\sfv)\; | \; \psi \in V^*, \sfv\in W\} \\
& =\partial (\s^{\otimes j}(W)) \\
& =\partial W.
\end{align*}
It now follows by induction that $\s^{\otimes j-i}(\partial ^i W)=\partial ^iW$ for all $i$. 

By definition, $\cD(W,i)=TV/(\pd^iW)$  so $\s$ is an automorphism of $\cD(W,i)$. 
\end{pf}

\begin{theorem}
 \label{thm.ws} 
With the notation in \S\ref{ssect.notn1},  
$$
\Aut(S)\;=\; \{\s \in \GL(V) \; | \; \s^{\otimes \ell}(k \sfw)=k \sfw\}.
$$  
\end{theorem} 
\begin{proof}  
Recall that $V=S_1$. Let $\s \in \GL(V)$.

Suppose $\s \in \Aut(S)$. By definition, $R=\ker(\xymatrix{V^{\otimes m} \ar[r]^{\rm mult} & S_m})$, so $\s^{\otimes m}(R)=R$.
Therefore $\s^{\otimes i}(V^{\otimes s}\otimes R\otimes V^{\otimes t})=V^{\otimes s}\otimes R\otimes V^{\otimes t}$ for all $s$ and $t$ such that 
$s+m+t=i$. Hence $\s^{\otimes i}(W_{i})=W_{i}$ for all $i$.
But $W_\ell=k\sfw$ so  $\s^{\otimes \ell}(k\sfw)=k\sfw$.

Conversely, suppose $\s^{\otimes \ell}(k\sfw)=k\sfw$. By Proposition \ref{prop.sdw}, $S=\cD(k\sfw, \ell-m)$, so $\s \in \Aut(S)$ by  Lemma 3.1.
\end{proof}

\begin{theorem} \label{thm.msss} 
Adopt the notation in \S\ref{ssect.notn1}. If $\s\in \Aut( S)$, then 
$$
\s^{\otimes \ell}(\sfw)=\hdet (\s)\sfw. 
$$ 
\end{theorem} 
\begin{proof}  
By Theorem \ref{thm.ws}, $\s^{\otimes \ell}(\sfw)=\l\sfw$ for some $0\neq \l\in k$. Hence $\s^{\otimes \ell}$ is multiplication by $\l$
on $W_{\ell}=k\sfw$.  Therefore, if $u \in W_{\ell}^*$, then $(\s^*)^{\otimes \ell}(u)=(\s^{\otimes \ell})^*(u)=\l u$.
But $W_{\ell}^* \cong \Ext^d_S(k,k)$ and, by Proposition \ref{prop.WZ},  $\hdet (\s)u=(\s^*)^{\otimes \ell}(u)$  for all $u \in \Ext^d_S(k,k)$. Therefore $\l=\hdet (\s)$.  
\end{proof}

\section{Nakayama Automorphisms} 
\label{sect.nak.aut}

We continue to use the notation and assumptions in  \S\ref{ssect.notn1}. 
This section proves several results about  the Nakayama automorphisms of   $S$ and $\Ext_S^*(k,k)$.  

\subsection{}
\label{ssect.nak.autom1}
The term ``Nakayama automorphism'' has different meanings in the literature. We defined the term for  Calabi-Yau algebras
in \S\ref{ssect.CY}  but there is an older usage. 

A {\sf Frobenius algebra} is a finite dimensional $k$-algebra $R$ such that $R \cong R^*$ 
as left $R$-modules. Let $\a:R \to R^*$ be a left $R$-module isomorphism. The  bilinear form $\langle \cdot,\cdot \rangle:R \times R \to k$
defined  by $\langle x,y \rangle =\a(y)(x)$ has the property that $\langle a,bc\rangle  
= \langle ab,c\rangle$ and, because $\a$ is injective, $\langle \cdot,\cdot \rangle$ is non-degenerate. 
 Because $\langle \cdot,\cdot \rangle$ is non-degenerate there is a unique linear map $\mu:R \to R$ such that $\langle a,b\rangle = \langle b,\mu(a)\rangle$ for all $a,b \in R$. The map $\mu$ is called a {\sf Nakayama automorphism} of $R$.  Because 
$\langle a,bc\rangle = \langle ab,c\rangle$,  $\mu$ is an algebra automorphism. 
Of course, $\mu$ depends on the choice of isomorphism $\a$.

\subsection{}
\label{ssect.nak.autom1.5}
A finite dimensional $\NN$-graded $k$-algebra $A$ is a {\sf graded Frobenius algebra} if there is an integer $\ell$ such that 
$A \cong A^*(-\ell)$ as graded left $A$-modules where $(A^*)_i=(A_{-i})^*$. 
Let $A$ be a connected graded Frobenius algebra and let $A_\ell$ be its top-degree non-zero component.
Then $\dim_k(A_\ell)=1$. Let $\sfv$ be a basis for $A_\ell$. 
There is, up to non-zero scalar multiples, a unique isomorphism $A \to A^*(-\ell)$ of graded left $A$-modules and hence a unique (up to non-zero scalar multiples) non-degenerate degree-preserving bilinear map $\langle \cdot,\cdot \rangle:A \times A \to k(-\ell)$ such that 
$\langle ab,c\rangle =\langle a,bc\rangle$ for all $a,b,c \in A$, namely
$$
\langle x,y \rangle = \begin{cases} 
                                       	\l & \text{if $x \in A_i$ and $y \in A_{\ell-i}$ and $xy=\l\sfv$}
					\\
					0 & \text{otherwise.}
				\end{cases}
$$
Although $\langle \cdot,\cdot \rangle$ depends on the choice of $\sfv$, the linear map $\mu:A \to A$ such that $\langle x,y \rangle
= \langle y,\mu(x) \rangle$ for all $x,y \in A$ does not. We call $\mu$ {\it the} 
{\sf Nakayama automorphism} of $A$.
  		
Since $\langle 1,\sfv \rangle = \langle \sfv,1\rangle$, $\mu(\sfv)=\sfv$. 

Let $x$ be a non-zero element in $A_1$. Then there is $y \in A_{\ell-1}$ such that $\langle x,y \rangle =1$, i.e., $xy=\sfv$. But 	
$\langle x,y \rangle	= \langle y,\mu(x) \rangle$ so $\sfv=xy=y\mu(x)$

\begin{prop}
\label{prop.mu.central}
Let $A$ be a connected graded Frobenius algebra. Its Nakayama automorphism belongs to the center of $\Aut(A)$.
\end{prop} 
\begin{pf}
The proof uses some of the observations in \cite[\S3]{S}. 

Let $\mu$ be the Nakayama automorphism and $\s \in \Aut(A)$. 
Let $A_n$ be the top-degree component of $A$ and $\sfv$ a basis for $A_n$. The bilinear form $\langle \cdot,\cdot \rangle$
on $A$ is, up to non-zero scalar multiple, given by the formula  $\langle x,y \rangle \sfv = xy$ for $x \in A_i$ and $y \in A_{n-i}$. 

Let $x \in A_i$ and $y \in A_{n-i}$. Since $xy=y\mu(x)$, 
$$
\s(y)\mu\s(x)=\s(x)\s(y)=\s(xy)=\s(y\mu(x))=\s(y)\s\mu(x),
$$
$\langle \s(y),\mu\s(x) \rangle=\langle\s(y),\s\mu(x) \rangle$. Since  $\langle \cdot,\cdot \rangle:A_{n-i}\times A_i\to k$ is non-degenerate, $\mu\s(x)=\s\mu(x)$.
Thus $\mu\s=\s\mu$. 
\end{pf}
				
\subsection{}
\label{ssect.nak.autom2} 
By \cite [Cor. 5.12]{BM}, $E$ is a graded Frobenius algebra.\footnote{This is a consequence of a more general
result in \cite{LPWZ} for $A_\infty$-algebras but the proof in \cite{BM} is more elementary.}

After the discussion in \S\S\ref{sse.cE} and \ref{ssect.nak.autom1},
the Frobenius pairing $\<\cdot,\cdot \>:E\times E\to k$ is given by the formula  
$$
\<\xi , \eta\> \; = \; (-1)^{i(d-i)}(\xi\eta)(\sfw) \qquad \hbox{for $\xi \in E^i$ and $\eta \in E^{d-i}$}
$$
and the Nakayama automorphism is the unique linear map $\mu:E \to E$ such that 
\begin{equation}
\label{eq.BM}
\<\xi, \eta\>=\<\eta, \mu(\xi)\>
\end{equation}
for all $\xi\in E^i$ and $\eta\in E^{d-i}$ and all $i$  (cf. \cite[p.91]{BM} which uses the notation $(\cdot,\cdot)$ 
where we use $\<\cdot,\cdot \>$, and they use  $\<\cdot,\cdot \>$ to denote the pairing between a vector space and its linear dual).

\begin{lemma} \label{lem.ndxe} 
 We adopt the notation and hypotheses in \S\ref{ssect.notn1}. If $\s\in \Aut( S)$, then 
 $$
\<\s^!(\xi), \s^!(\eta)\> \; =\; \hdet (\s)\<\xi, \eta\>
$$
for all $\xi\in E^i$, all $\eta\in E^{d-i}$, and all $i=0,\ldots,d$. 
\end{lemma} 

\begin{proof}  
The calculation
\begin{align*}
\<\s^!(\xi), \s^!(\eta)\> 
& =(-1)^{i(d-i)}(\s^!(\xi)\s^!(\eta))(\sfw)
\\
& =(-1)^{i(d-i)}(\s^!(\xi\eta))(\sfw) \\
& =(-1)^{i(d-i)}(\xi \eta)(\s^{\otimes \ell}(\sfw))
\\
&=(-1)^{i(d-i)}(\xi\eta)(\hdet (\s)\sfw) \qquad \hbox{(by Theorem \ref{thm.msss})}
\\
& =\hdet (\s)(-1)^{i(d-i)}(\xi \eta)(\sfw)
\\
&=\hdet (\s)\<\xi, \eta\>
\end{align*}
proves the lemma.
\end{proof}

The next result follows from \cite [Thm. 6.3]{BM}.

\begin{lemma}
 \label{lem.BM} 
 With the notation in \S\ref{ssect.notn1},  $\mu|_{V^*}=(-1)^{d+1}\nu^*$. 
\end{lemma}

\begin{theorem} \label{thm.mnw} 
We adopt the notation in \S\ref{ssect.notn1}. 
The Nakayama automorphism of $S$ is the unique $\nu \in \GL(V)$ such that
$$
(\nu \otimes \id^{\otimes \ell-1})\phi (\sfw)=(-1)^{d+1}\sfw.
$$
In particular, $\sfw$ is a $(-1)^{d+1}\nu$-twisted superpotential.  
\end{theorem} 

\begin{proof} 
Let $\nu$ be the Nakayama automorphism of $S$. 
If  $\xi\in E^1$ and $\eta\in E^{d-1}$,  then 
\begin{align*}
(\xi \eta)(\sfw)
& =(-1)^{d-1}\<\xi, \eta\>
\\
&=(-1)^{d-1}\<\eta, \mu(\xi)\> \\
& =(\eta\mu(\xi))(\sfw)
\\
& =(\mu(\xi)\eta)(\phi (\sfw))  
\\
& =((\mu \otimes \id^{\otimes \ell-1})(\xi\eta))(\phi (\sfw)).
\end{align*}
By Lemma \ref{lem.BM}, $(\mu|_{V^*})^*=(-1)^{d+1}\nu$ so 
$$
(\mu \otimes \id^{\otimes \ell-1})(\xi\eta) = (\xi\eta)\circ ((-1)^{d+1}\nu \otimes \id^{\otimes \ell-1}).
$$
Therefore $(\nu \otimes \id^{\otimes \ell-1})(\phi (\sfw))=(-1)^{d+1}\sfw$.   

Conversely, suppose that $\nu' \in \Aut(S)\subset \GL(V)$ satisfies $(\nu' \otimes \id^{\otimes \ell-1})\phi (\sfw)=(-1)^{d+1}\sfw$.  
Let $\mu':=(-1)^{d+1}(\nu')^*$.  If $\xi\in E^1$ and $\eta\in E^{d-1}$, then
\begin{align*}
\<\eta, \mu'(\xi)\> 
& =(-1)^{d-1}(\eta\mu'(\xi))(\sfw)
\\
& =(-1)^{d-1}(\mu'(\xi)\eta)(\phi (\sfw))
 \\
& =(-1)^{d-1}((\mu'\otimes \id^{\otimes \ell-1})(\xi\eta))(\phi (\sfw)) \\
& =(-1)^{d-1}(\xi\eta)((-1)^{d+1}\nu' \otimes \id^{\otimes \ell-1})(\phi (\sfw)) \\
& =(-1)^{d-1}(\xi \eta)(\sfw)
\\
&=\<\xi, \eta\>.
\end{align*}
But $\langle \xi,\eta\rangle = \langle \eta,\mu(\xi)\rangle$ so $\mu'=\mu|_{V^*}$. However, by Lemma \ref{lem.BM}, $\mu|_{V^*}=(-1)^{d+1}\nu^*$ 
so 
$$
 (-1)^{d+1}(\nu')^* = \mu' = \mu|_{V^*}=(-1)^{d+1}\nu^*.
 $$
Thus, as elements of $\GL(V)$, and hence as elements in $\Aut(S)$, $\nu'=\nu$.  
\end{proof} 

\begin{corollary}
 \label{cor.sp} 
With the notation in \S\ref{ssect.notn1}, $S$ is Calabi-Yau if and only if $\phi (\sfw)=(-1)^{d+1}\sfw$. 
\end{corollary} 

\begin{remark} \label{rem.sp} 
Let $S$ be an $m$-Koszul AS-regular algebra of Gorenstein parameter $\ell$.  For $m=2$, $S$ is Calabi-Yau if and only if $\phi (\sfw)=(-1)^{\ell+1}\sfw$ (a superpotential in the sense of \cite {BSW}) by \cite [Thm. 6.2]{BSW},  and, for $m>2$, $S$ is Calabi-Yau if and only if $\phi (\sfw)=\sfw$ (a superpotential in our sense) by \cite [Thm. 6.8]{BSW}. (If $m \ge 3$, then $d$ is odd.)
The above corollary combines these results and makes them less mysterious.  
(Note that one direction of the above corollary was essentially proved in \cite [Thm. 4.12]{WZ}.)  
\end{remark}

\begin{example}
Let $V$ be a vector space with basis $\{x,y\}$.  Let 
$$
\sfw \; = \;  x^2y^2+yx^2y+y^2x^2+xy^2x \; \in \; V^{\otimes 4}.
$$
The algebra
$$
S\; = \; \cD(\sfw, 1) \; = \; \frac{k\<x, y\>}{(xy^2+y^2x, x^2y+yx^2)}
$$
  is a 3-Koszul AS-regular algebra of dimension 3 with Gorenstein parameter 4.  By \cite [Example 1.6]{RRZ}, $S$ is Calabi-Yau.  Since $\phi (\sfw)=\sfw$, $\sfw$ is a superpotential in our sense, but not a superpotential in the sense of \cite{BSW}. 
\end{example}

\subsection{}
Let $S$ be an $m$-Koszul AS-regular algebra and $\s\in \Aut(S)$.  The next result gives a formula for the Nakayama 
automorphism $\mu$ of $E$ in terms of that for $S$, which we denote by $\nu$.  
When $m=2$, i.e., when $S$ is a Koszul algebra, then $E$ is generated in degree 1 so $\mu$ is determined by its restriction to $E_1=V^*$
so, by Lemma \ref{lem.BM} (\cite [Thm. 6.3]{BM}), $(-1)^{d+1}\nu^*$.  When $m\ge 3$, 
$E$ is not generated in degree 1, so $\mu$ is not determined by its restriction to $E_1$. Hence the need for the somewhat
awkward calculations in the next proof. 

\begin{theorem}
\label{thm.BM.nak} 
With the notation in \S\ref{ssect.notn1}, $\mu=(\e^{d+1}\nu)^!$. 
\end{theorem} 
\begin{proof}  
Let $\xi \in E^i$ and $\eta\in E^{d-i}$. Then
\begin{align*}
(-1)^{i(d-i)}\big\langle  \eta, (\e^{d+1}\nu)^!(\xi)\big\rangle &  \;=\;    \big[\eta (\e^{d+1}\nu)^!(\xi)\big](\sfw) 
\\
&  \; = \; \big[(\e^{d+1}\nu)^!(\xi)\eta\big]\big(\phi ^{\rho (i)}(\sfw)\big)
\\
&  \; = \;  (-1)^{(d+1)\rho (i)} \big[ \nu^!(\xi) \eta\big]\big(\phi ^{\rho (i)}(\sfw)\big).
 \end{align*}
Repeated use of Theorem \ref{thm.mnw} and an induction argument on $j$ shows that 
$$
\phi ^j(\sfw)=(-1)^{(d+1)j}\Big((\nu^{-1})^{\otimes i}\otimes \id^{\otimes \ell-j}\Big)(\sfw)
$$
for $0\leq j\leq \ell$.  In particular, 
$$
\phi ^{\rho(i)}(\sfw)=(-1)^{(d+1)\rho(i)}\Big((\nu^{-1})^{\otimes \rho(i)}\otimes \id^{\otimes \ell-\rho(i)}\Big)(\sfw)
$$
so
\begin{align*}
(-1)^{i(d-i)}\big\langle  \eta, (\e^{d+1}\nu)^!(\xi)\big\rangle & \; = \;   (-1)^{(d+1)\rho (i)} \big[ \nu^!(\xi) \eta\big]\big(\phi ^{\rho (i)}(\sfw)\big)
\\
 & \; = \;     \big[ \nu^!(\xi) \eta\big]\bigg(\Big((\nu^{-1})^{\otimes \rho(i)}\otimes \id^{\otimes \ell-\rho(i)}\Big)(\sfw) \bigg)
\end{align*}
which equals
$$
 \bigg[\Big((\nu^*)^{\otimes \rho (i)}\otimes  \id^{\otimes \ell-\rho (i)}\Big)(\xi \eta)\bigg] \bigg(\Big((\nu^{-1})^{\otimes \rho(i)}\otimes \id^{\otimes \ell-\rho(i)}\Big)(\sfw) \bigg).
 $$
This, in turn, is equal to 
 $$
 (\xi \eta)
\Big(\big[\nu^{\otimes \rho (i)}\otimes \id^{\otimes \ell-\rho (i)}\big]
\big[(\nu ^{-1})^{\otimes \rho (i)}\otimes \id^{\otimes \ell-\rho (i)}\big](\sfw)\Big).
$$
But this is equal to  $(\xi \eta)(\sfw)$ which is $(-1)^{i(d-i)}\<\xi, \eta\>$. Thus, we have shown that
$$
(-1)^{i(d-i)}\big\langle  \eta, (\e^{d+1}\nu)^!(\xi)\big\rangle  \; = \; (-1)^{i(d-i)}\<\xi, \eta\>
$$
from which it follows that  $(\e^{d+1}\nu )^!$ is the Nakayama automorphism of $E$.   
\end{proof} 

A more general case of the following theorem was proved in \cite [Thm.4.2]{LMZ} using Hochschild cohomology 
(cf. \cite [Thm.3.11]{RRZ} for the noetherian case).  The following is a different proof using Proposition \ref{prop.mu.central}.  

\begin{corollary} \label{cor.mnc}  
 The Nakayama automorphism of an $m$-Koszul AS-regular algebra $S$ belongs to the center of $\Aut(S)$.
\end{corollary} 

\begin{proof} 
Let $\s\in \Aut(S)$. Since $(\s\nu)^!=\nu^!\s^!$ and $\mu$ commutes with $\s^!$,
$$
(\e^{d+1}\s\nu)^! =(\e^{d+1}\nu)^!\s^! 
 = \mu \s^!  = \s^!\mu 
 =\s^!(\e^{d+1}\nu)^!=(\e^{d+1}\nu\s)^!.
$$
Hence $\s\nu =\nu \s$ for all $\s\in \Aut(S)$. 
\end{proof} 

The next lemma follows from \cite[Prop. 5.3]{BM}. 

\begin{lemma} \label{lem.dl}  
\label{lem.mek} 
Adopt the notation in \S\ref{ssect.notn1}. If $m=2$, then $S$ is Koszul in the ``classical'' sense and $\ell=d$. If $m \ge 3$, then $d$ is odd and $\ell=\frac{1}{2}m(d-1)+1$. 
\end{lemma} 

\subsection{}
 
Suppose for a moment that $\nu$ is the Nakayama automorphism of a noetherian connected graded AS-regular algebra.
Reyes-Rogalski-Zhang have proved that $\hdet(\nu)=1$  \cite[Thm. 5.3]{RRZ2}. This improved their earlier result which 
required $A$ to also be Koszul \cite[Thm.6.3]{RRZ}. The next result extends this to the $m$-Koszul case 
without the noetherian hypothesis.

\begin{theorem} \label{thm.hdn}
With the notation in \S\ref{ssect.notn1},  $\hdet (\nu)=1$. 
\end{theorem} 

\begin{proof}   
Let $\xi\in E^i$ and $\eta\in E^{d-i}$.  
By Theorem \ref{thm.BM.nak},   $\mu=(\e^{d+1}\nu)^!$ so
\begin{align*}
\<\xi, \eta \> 
 =\<\eta, \mu(\xi)\> &=\<\mu(\xi), \mu(\eta)\> \\
& =\<(\e^{d+1}\nu)^!(\xi), (\e^{d+1}\nu)^!(\eta)\> \\
& =\hdet (\e^{d+1}\nu)\<\xi, \eta\>
\end{align*}
where the last equality is given by Lemma \ref{lem.ndxe}.
Hence $\hdet (\e^{d+1}\nu)=1$. Now, by Theorem \ref{thm.msss}, 
$$
\sfw  \; = \;   \hdet (\e^{d+1}\nu)\sfw \; = \; (\e^{d+1}\nu)^{\otimes \ell}(\sfw).
$$
Therefore
$$
\sfw =((-1)^{d+1}\nu)^{\otimes \ell}(\sfw)  =(-1)^{(d+1)\ell}\nu^{\otimes \ell}(\sfw)=(-1)^{(d+1)\ell}\hdet (\nu)\sfw.
$$
Hence $\hdet (\nu)=(-1)^{(d+1)\ell}$.  But either $d$ is odd or $\ell$ is even by Lemma \ref{lem.mek}
so $\hdet (\nu)=(-1)^{(d+1)\ell}=1$. 
\end{proof}

Since Theorem \ref{thm.hdn} does not need the noetherian hypothesis it might be reasonable to drop the noetherian assumption
 in \cite [Conjecture 6.4]{RRZ}.

\begin{conjecture} 
If $\nu$ is  the Nakayama automorphism of an AS-regular algebra,  then $\hdet (\nu)=1$. 
\end{conjecture} 

\subsection{}
Let $x_1, \dots, x_n$ be a basis for $V$, and ${\bf x}^t=(x_1, \dots, x_n)$.  For every $\sfw\in V^{\otimes \ell}$, there exists a unique $n\times n$ matrix $M$  whose entries are in $V^{\otimes \ell-2}$ such that $\sfw={\bf x}^tM{\bf x}$ with the usual matrix multiplication.  
 The $d=3$ case of the following result is proved in \cite [Corollary 9.3]{V} and \cite [Prop. 6.5]{BM}. 

\begin{theorem} \label{thm.mnw2} 
Adopt the notation in \S\ref{ssect.notn1}. 
Choose a basis $x_1, \dots, x_n$ for $V=S_1$ and define ${\bf x}^t=(x_1,\ldots,x_n) \in V^n$.
 Let $M$ be the unique $n\times n$ matrix with entries in $V^{\otimes \ell-2}$ such that
$\sfw={\bf x}^tM{\bf x}$. There is a unique matrix $Q\in \GL(n, k)$ such that $({\bf x}^tM)^t=QM{\bf x}$.
 With respect to the ordered basis $x_1, \dots, x_n$, the Nakayama automorphism of $S$ is
 $$
 \nu =(-1)^{d+1}Q^{-t}\in \GL(n, k).
 $$
 Thus, $S$ is Calabi-Yau if and only if $Q=(-1)^{d+1}\Id$. 
\end{theorem} 

\begin{proof}  
By \cite [Prop. 2.11]{AS},  $({\bf x}^tM)^t=QM{\bf x}$ if and only if $(Q^{-t}\otimes \id_V ^{\otimes \ell-1})(\phi (\sfw))=\sfw$, so the result follows from Theorem \ref{thm.mnw}. 
\end{proof} 

\begin{example} 
Let $S=k[x, y]$ be the polynomial algebra in two variables of degree 1. Since
$$
\sfw=xy-yx=\begin{pmatrix} x & y \end{pmatrix}\begin{pmatrix} 0 & 1 \\ -1 & 0 \end{pmatrix}\begin{pmatrix} x \\ y \end{pmatrix},
$$  
 the matrix $Q$ in Theorem \ref{thm.mnw2} is $-\Id$.  Since $S$ is Calabi-Yau, its  Nakayama automorphism is 
the identity. Therefore 
$$
(\nu \otimes \id)\phi (\sfw)=(\id\otimes \id)(yx-xy)=-\sfw
$$
in accordance with the factor $(-1)^{d+1}$ that appears in Theorems \ref{thm.mnw} and  \ref{thm.mnw2} since $d=\gldim(S)=2$. 
\end{example}

\section{Twists } 
\label{S.twists}

In this section, we examine the relation between twisting a superpotential and twisting an algebra by an automorphism \`a la Zhang
\cite{Z}.

\subsection{}  
Let $\s\in \GL(V)$. Let $\sfw \in V^{\otimes \ell}$ and let $W$ be a subspace of $V^{\otimes \ell}$. We define 
$$
\sfw^{\s}:=(\s^{\ell-1}\otimes \cdots  \otimes\s\otimes \id_V)(\sfw).
$$
and
$$
W^{\s}:=(\s^{\ell-1}\otimes \cdots \otimes \s\otimes \id_V)(W) = \{\sfw^\s \; | \; \sfw \in W\}.
$$

\subsection{}
Let $A$ be a graded algebra and $\s\in \Aut(A)$. The {\sf Zhang twist} of $A$ by $\s$ is the graded vector space $A^{\s}:=A$ with  multiplication $a*b:=a\s^i(b)$ for $a\in A_i, b\in A$.\footnote{The paper \cite {RRZ} uses the convention $a*b=\s^j(a)b$ for $a\in A, b\in A_j$.  The twist of $A$ by $\s\in \Aut(A)$ in the sense of \cite {RRZ} is the same as the twist of $A$ by $\s^{-1}\in \Aut(A)$ in our sense.}

\begin{lem} 
Let $R$ be a subspace of $V^{\otimes m}$ and let $A=TV/(R)$. If $\s \in \Aut(A)$, then   $A^{\s} \cong TV/(R^{\s})$ . 
 \end{lem}

\begin{proposition} \label{prop.q3}
Let $W\subseteq   V^{\otimes \ell}$ be a subspace and $\s\in \GL(V)$. 
\begin{enumerate} 
\item{} $\partial ^i(W^{\s})=(\partial ^iW)^{\s}$.  
\item{} If $\cD(W, i)=TV/(R)$ where $R\subseteq   V^{\otimes \ell-i}$, then $\cD(W^{\s}, i)=TV/(R^{\s})$.  
\item{} If $\s \in \Aut \big(\cD(W, i)\big)$, then $\cD(W^{\s}, i)\cong \cD(W, i)^{\s}$. 
\end{enumerate} 
\end{proposition} 

\begin{proof}
(1)  
Let $\id$ denote the identity automorphism of $V$. Since 
\begin{align*}
\partial (W^{\s})
& =\{(\psi \otimes \id^{\otimes \ell-1})(\s^{\ell-1}\otimes \s^{\ell-2}\otimes \cdots \otimes \id)(\sfw)\; | \; \psi\in V^*, \sfw\in W\} \\
& =\{(\psi \s^{\ell-1}\otimes \s^{\ell-2}\otimes \cdots \otimes \id)(\sfw)\; | \; \psi\in V^*, \sfw\in W\} \\
& =\{(\psi \otimes \s^{\ell-2}\otimes \cdots \otimes \id)(\sfw)\; | \; \psi\in V^*, \sfw\in W \} \\ 
& =\{(\s^{\ell-2}\otimes \cdots \otimes \id)(\psi \otimes \id^{\otimes \ell-1})(\sfw)\; | \; \psi\in V^*, \sfw\in W\} \\
& =(\partial W)^{\s},
\end{align*}
an induction argument shows that  $\partial ^i(W^{\s})=(\partial ^iW)^{\s}$.  

(2) If $\cD(W, i)=TV/(R)$, then $R=\partial ^i(W)$, so 
$$\cD(W^{\s}, i)=TV/(\partial ^i(W^{\s}))=TV/((\partial ^iW)^{\s})=TV/(R^{\s})$$
by (1). 

(3) If $\cD(W, i)=TV/(R)$ and $ \s\in \Aut\big( \cD(W, i)\big)$, then 
$$\cD(W^{\s}, i)=TV/(R^{\s}) \cong \big(TV/(R)\big)^{ \s}=\cD(W, i)^{\s}$$ by (2).
\end{proof} 

In the next result, the Nakayama automorphism of $TV/(R)$
is viewed as an element of $\GL(V)$. 

\begin{proposition} \label{prop.wst} 
Adopt the notation in \S\ref{ssect.notn1}. 
If  $ \s\in \Aut(S)$,  then $\sfw^{\s}$ is a $(-1)^{d+1}(\hdet  \s)^{-1}\s^{\ell}\nu $-twisted superpotential. 
\end{proposition} 

\begin{proof}   
By Theorem \ref{thm.msss}, $\s^{\otimes \ell}(\sfw)=\hdet ( \s)\sfw$.
By  Theorem \ref{thm.mnw},
$\phi (\sfw)=(-1)^{d+1}(\nu^{-1}\otimes \id^{\otimes \ell-1})(\sfw)$.  
Since 
\begin{align*}
 ((-1)^{d+1}&(\hdet  \s)^{-1}\s^{\ell}\nu \otimes \id^{\otimes \ell-1})\phi (\sfw^{\s}) \\
= & (-1)^{d+1}(\hdet  \s)^{-1}(\s^{\ell}\nu \otimes \id^{\otimes \ell-1})\phi ((\s^{\ell-1}\otimes \cdots \otimes \s \otimes \id)(\sfw))  \\
= & (-1)^{d+1}(\hdet  \s)^{-1}(\s^{\ell}\nu \otimes \id^{\otimes \ell-1})(\id\otimes \s^{\ell-1}\otimes \cdots \otimes \s)\phi (\sfw)  \\
= & (-1)^{d+1}(\hdet  \s)^{-1}(\s^{\ell}\nu \otimes \s^{\ell-1}\otimes \cdots \otimes \s)(-1)^{d+1}(\nu^{-1}\otimes \id^{\otimes \ell-1})(\sfw) \\
= & (\hdet  \s)^{-1}(\s^{\ell-1} \otimes \s^{\ell-2}\otimes \cdots \otimes \id)(\s\otimes \s\otimes \cdots \otimes \s)(\sfw)  \\
= & (\hdet  \s)^{-1}(\s^{\ell-1} \otimes \s^{\ell-2}\otimes \cdots \otimes \id)((\hdet  \s)\sfw) \\
= & (\s^{\ell-1} \otimes \s^{\ell-2}\otimes \cdots \otimes \id)(\sfw) \\
= & \sfw^{\s},
\end{align*}
$\sfw^{\s}$ is a twisted superpotential for $(-1)^{d+1}(\hdet  \s)^{-1}\s^{\ell}\nu$. 
\end{proof} 

In the next result, $\cD(W, i)$ is of the form $TV/(R)$ so $\Aut\big(\cD(W, i)\big) \subseteq \GL(V)$.
 
\begin{proposition} \label{prop.swi}
Let $W\subseteq  V^{\otimes \ell}$. For all $i=0,\ldots,\ell-1$,   $\s\in \Aut\big(\cD(W^\s, i)$.
\end{proposition} 

\begin{proof}
Since $\cD(W, i):=TV/(\partial^iW)$,  $\s^{\otimes \ell-i}(\partial ^iW)=\partial ^iW$.  
Since 
\begin{align*}
\s^{\otimes \ell-i}(\partial ^i(W^{\s}))
& = \s^{\otimes \ell-i}((\partial ^iW)^{\s}) \\
& =(\s\otimes \cdots \otimes \s)(\s^{\ell-i-1}\otimes \cdots \otimes \id)(\partial ^iW) \\
& =(\s^{\ell-i-1}\otimes \cdots \otimes \id)(\s\otimes \cdots \otimes \s)(\partial ^iW) \\
& =(\s^{\ell-i-1}\otimes \cdots \otimes \id)(\partial ^iW) \\
& =(\partial ^iW)^{\s} \\
& =\partial ^i(W^{\s}), 
\end{align*}
by Proposition \ref{prop.q3} (1), $\s$ extends to a graded algebra automorphism of $\cD(W^{\s}, i):=TV/(\partial ^i(W^{\s}))$.  
\end{proof}  

 The next result is a non-noetherian version of \cite[Thm. 5.4(a),(b)]{RRZ}. 

\begin{theorem} \label{thm.msn} 
Adopt the notation in \S\ref{ssect.notn1}. Let $\s\in \Aut(S)$.  
\begin{enumerate}
\item{} 
$S^{\s}$ is an $m$-Koszul AS-regular algebra of dimension $d$ and Gorenstein parameter $\ell$. 
\item{} $\sfw_{S^{\s}}=(\sfw_S)^{\s}$.  
\item{}  $\s\in \Aut(S^{\s})$ and $\hdet _{S^{\s}}(\s)=\hdet _{S}(\s)$.  
\item{} $\nu \in \Aut(S^{\s})$ and the Nakayama automorphism of $S^{\s}$ is 
$\hdet (\s)^{-1}\s^{\ell}\nu$.
\end{enumerate}
\end{theorem} 

\begin{proof}  
(1) 
The minimal free resolution of $k$ over $S^{\s}$ is obtained by twisting the minimal free resolution of $k$ over $S$ by $\s$,
so $S^{\s}$ is  also an $m$-Koszul AS-regular algebra of dimension $d$ and Gorenstein parameter $\ell$ (cf. \cite [Prop. 13]{DV}). 

(2) 
Let $\sfw=\sfw_S$.
Since $S=\cD(\sfw,\ell-m)$, $S^{\s}=\cD(\sfw, \ell-m)^{\s}$. But $\cD(\sfw, \ell-m)^{\s} \cong \cD(\sfw^{\s}, \ell-m)$ by Proposition \ref{prop.q3}(3)
so, by Proposition \ref{prop.sdw},  $\sfw_{S^{\s}}=\sfw^{\s}$.

(3)  
By Proposition \ref{prop.swi}, $\s$ is a graded $k$-algebra automorphism of $\cD(\sfw^{\s}, \ell-m)$. Hence $\s \in \Aut(S^\s)$.
  
By Theorem \ref{thm.msss}, $\s^{\otimes \ell}(\sfw)=\hdet_S(\s)\sfw$, so  
\begin{align*}
\s^{\otimes \ell}(\sfw_{S^{\s}}) 
& =\s^{\otimes \ell}(\sfw^{\s}) \\
& =\s^{\otimes \ell}(\s^{\ell-1}\otimes \cdots \otimes \id)(\sfw) \\
& =(\s^{\ell-1}\otimes \cdots \otimes \id)\s^{\otimes \ell}(\sfw) \\
& =(\s^{\ell-1}\otimes \cdots \otimes \id)(\hdet_S(\s) \sfw) \\
& =\hdet_S(\s)\sfw^{\s} \\
& =\hdet_S(\s) \sfw_{S^{\s}} 
\end{align*}
by (2), so $\hdet _{S^{\s}}(\s)=\hdet _{S}(\s)$ by Theorem \ref{thm.msss}. 

(4) Since $\nu \s=\s\nu $ by Corollary \ref{cor.mnc}, and $\hdet (\nu )=1$ by Theorem \ref{thm.hdn},
\begin{align*}
\nu ^{\otimes \ell}(\sfw_{S^{\s}}) 
& = \nu^{\otimes \ell}(\sfw^{\s}) \\
& = \nu^{\otimes \ell}(\s^{\ell-1}\otimes \cdots \otimes \id)(\sfw) \\
& = (\s^{\ell-1}\otimes \cdots \otimes \id)(\nu^{\otimes \ell})(\sfw) \\
& = (\s^{\ell-1}\otimes \cdots \otimes \id)(\hdet (\nu)\sfw) \\
& = (\s^{\ell-1}\otimes \cdots \otimes \id)(\sfw) \\
& = \sfw^{\s}\\
& =\sfw_{S^{\s}}
\end{align*} 
by (2) and Theorem \ref{thm.msss}.
Thus, by Theorem \ref{thm.ws}, $\nu\in \Aut(S^{\s})$. 
By Theorem \ref{thm.mnw} and Proposition \ref{prop.wst}, 
$\nu _{S^{\s}}=(\hdet( \s))^{-1}\s^{\ell}\nu$ is the Nakayama automorphism of $S^{\s}$.
\end{proof}

\begin{corollary} \label{cor.ocy}

Adopt the notation in \S\ref{ssect.notn1} and suppose, further, that $S$ is Calabi-Yau.
 If $\s\in \Aut(S)$, then $S^{\s}$ is Calabi-Yau if and only if $\s^{\ell}=\hdet (\s)\id_S$. 
\end{corollary} 

\begin{proof} Since $\nu=\id_S$,  $S^{\s}$ is Calabi-Yau if and only if $\hdet( \s)^{-1}\s^{\ell}=\nu _{S^{\s}}=\id_{S^{\s}}$ by Theorem \ref{thm.msn} if and only if $\s^{\ell}=\hdet (\s)\id_S$. 
\end{proof} 

\begin{prop} 
Let $S$ be the polynomial ring on three indeterminates. Then $S^{\s}$ is Calabi-Yau if and only if 
$$
S^\s \; \cong \; \frac{k\<x, y, z\>}{(yz-\xi zy, zx-\xi xz, xy-\xi yx)}
$$ 
for some $\xi \in k$ such that $\xi^3=1$. 
\end{prop}
\begin{pf} 
Suppose $\s$ is in Jordan normal form with respect to the basis $\{x,y,z\}$ for $S_1=V$.  
Since $S$ is Calabi-Yau, $S^{\s}$ is Calabi-Yau if and only if $\s^3=\hdet (\s)\id_V=\det (\s)\id_V$. 
Clearly, $\s^3=\det (\s)\id_V$ if and only if  $\s={\rm diag}(\a, \a\xi,\a\xi^2)$ for some $\a,\xi \in k-\{0\}$ such that  
$\xi^3=1$. For such a $\s$, $S^\s$ is as claimed.
\end{pf}

\section{Jacobian algebras}
\label{sect.Jac}

Throughout \S\ref{sect.Jac}, we assume  that the characteristic of $k$ does not divide $\ell$.

Define $c:V^{\otimes \ell}\to V^{\otimes \ell}$ 
by 
$$
c(\sfw) \; :=\; \frac{1}{\ell}\, \sum _{i=0}^{\ell-1}\phi ^i(\sfw).
$$
The {\sf $i$-th order Jacobian algebra} of $\sfw$  is  $J(\sfw, i)  :=\cD(c(\sfw), i)$

\begin{lemma}
 \label{lem.cscs1}
The set of superpotentials in $V^{\otimes \ell}$ is equal to the image of $c$, i.e.,  
$\Im( c)=\{\sfw\in V^{\otimes \ell}\; | \; \phi(\sfw)=\sfw\}$. 
\end{lemma}

\begin{proof} If $\phi (\sfw)=\sfw$, then 
$$c(\sfw)=\frac{1}{\ell}\sum _{i=0}^{\ell-1}\phi ^i(\sfw)=\frac{1}{\ell}\sum _{i=0}^{\ell-1}\sfw=\sfw,$$
so $\sfw=c(\sfw)\in \Im( c)$.  
For the converse, since $\phi ^{\ell}(\sfw)=\sfw$, if $c(\sfw)\in \Im (c)$, then 
$$\phi (c(\sfw))=\phi \left (\frac{1}{\ell}\sum _{i=0}^{\ell-1}\phi ^i(\sfw)\right)=\frac{1}{\ell}\sum_{i=0}^{\ell-1}\phi ^{i+1}(\sfw)=c(\sfw).$$
\end{proof}

\begin{corollary} \label{cor.cy2} 
Let $0\neq \sfw\in V^{\otimes \ell}$.  If $S=J(\sfw, \ell-m)$ is an $m$-Koszul Calabi-Yau algebra, then  $\sfw_S=c(\sfw)$ 
up to a non-zero scalar multiple.  
\end{corollary} 

\begin{proof} 
If $S=J(\sfw, \ell-m):=\cD(c(\sfw), \ell-m)$ is Calabi-Yau, then it is AS-regular by Lemma \ref{lem.RRZ}.  Since $c(\sfw)$ is a (twisted) superpotential by Lemma \ref{lem.cscs1},  $c(\sfw)=\sfw_S$ up to a non-zero scalar multiple by Proposition \ref{prop.sdw}.
\end{proof} 
 
\begin{theorem} \label{thm.cycp}
Adopt the notation in \S\ref{ssect.notn1}.  
If $d$ is odd, then the following are equivalent: 
\begin{enumerate}
\item{} $S$ is Calabi-Yau. 
\item{} $\phi(\sfw)=\sfw$. 
\item{} $c(\sfw)=\sfw$. 
\item{} $S\cong J(\sfw, \ell-m)$. 
\end{enumerate}
\end{theorem} 

\begin{proof}  
  (1) $\Leftrightarrow $ (2).
  By Corollary \ref{cor.sp},  $S$ is Calabi-Yau if and only if $\phi (\sfw)=(-1)^{d+1}\sfw$.
Since $d$ is odd, $S$ is Calabi-Yau if and only if $\phi (\sfw)=\sfw$

 (2) $\Leftrightarrow $ (3).
If $\phi(\sfw)=\sfw$, then $c(\sfw)=\sfw$.   
Conversely, if $c(\sfw)=\sfw$, then $\phi (\sfw)=\sfw$ by Lemma \ref{lem.cscs1}.  

 (3) $\Leftrightarrow $ (4). 
If $c(\sfw)=\sfw$, then $S=\cD(\sfw, \ell-m)=\cD(c(\sfw), \ell-m)=:J(\sfw, \ell-m)$ by Proposition \ref{prop.sdw}.  Conversely, 
since $S$ is $m$-Koszul AS-regular and $c(\sfw)$ is a (twisted) superpotential by Lemma \ref{lem.cscs1}, if $S=J(\sfw, \ell-m):=\cD(c(\sfw), \ell-m))$, then $\sfw=c(\sfw)$ by Proposition \ref{prop.sdw}
\end{proof} 

\begin{theorem}
Adopt the notation in \S\ref{ssect.notn1}.    If $d$ is even,  then the following are equivalent: 
\begin{enumerate}
\item{} $S$ is Calabi-Yau. 
\item{} $\phi(\sfw)=-\sfw$. 
\item{} $\widetilde c(\sfw)=\sfw$. 
\item{} $S\cong\widetilde J(\sfw, \ell-m)$. 
\end{enumerate} where $\widetilde c(\sfw):=\frac{1}{\ell}\sum _{i=0}^{\ell-1}(-1)^i\phi ^i(\sfw)$ and $\widetilde J(\sfw, i):=\cD(\widetilde c(\sfw), i)$.  
\end{theorem} 

\begin{proof} 
 (1) $\Leftrightarrow $ (2).
  By Corollary \ref{cor.sp},  $S$ is Calabi-Yau if and only if $\phi (\sfw)=(-1)^{d+1}\sfw$.
Since $d$ is even, $S$ is Calabi-Yau if and only if $\phi (\sfw)=-\sfw$

 (2) $\Rightarrow $ (3).
If $\phi(\sfw)=-\sfw$, then 
$$\widetilde c(\sfw)=\frac{1}{\ell}\sum _{i=0}^{\ell-1}(-1)^i\phi ^i(\sfw)=\frac{1}{\ell}\sum _{i=0}^{\ell-1}(-1)^i(-1)^i(\sfw)
=\sfw.$$

(3) $\Rightarrow $ (2).
Suppose $\widetilde c(\sfw)=\sfw$.
Since $d$ is even, Lemma \ref{lem.dl} tells us that $S$ is Koszul and $\ell=d$.
 Therefore  $(-1)^{\ell}\phi ^{\ell}(\sfw)=\sfw$. Hence 
\begin{align*}
\phi (\sfw)  =\phi (\widetilde c(\sfw)) & =\phi \left(\frac{1}{\ell}\sum _{i=0}^{\ell-1}(-1)^i\phi ^i(\sfw)\right) \\
& =\frac{1}{\ell}\sum _{i=0}^{\ell-1}(-1)^i\phi ^{i+1}(\sfw)
\\
& =-\frac{1}{\ell}\sum _{i=0}^{\ell-1}(-1)^{i+1}\phi ^{i+1}(\sfw) \\
& =-\frac{1}{\ell}\sum _{i=0}^{\ell-1}(-1)^i\phi ^i(\sfw)
\\
&=-\widetilde c(\sfw)=-\sfw.
\end{align*}  

 (3) $\Leftrightarrow $ (4).
If $\widetilde c(\sfw)=\sfw$, then $S=\cD(\sfw, \ell-m)=\cD(\widetilde c(\sfw), \ell-m)=:\widetilde J(\sfw, \ell-m)$ by Proposition \ref{prop.sdw}.  For the converse, since $\phi (\widetilde c(\sfw))=-\widetilde c(\sfw)$ by the above proof, $\widetilde c(\sfw)$ is a twisted superpotential.  
Since $S$ is $m$-Koszul AS-regular, if $S=\widetilde J(\sfw, \ell-m):=\cD(\widetilde c(\sfw), \ell-m))$, then $\sfw=\widetilde c(\sfw)$ by Proposition \ref{prop.sdw}
\end{proof}

\section{AS-regular Algebras of dimensions 2 and 3}

Let $S=TV/(R)$ be an AS-regular algebra. 
Frequently, $\hdet (\s)=\det (\s)$ for all  $\s\in \Aut(S)$. In this section we examine some (non-)examples of this phenomenon.

\subsection{} 
Let $S$ be a noetherian AS-regular algebra of dimension 2 generated in degree 1.
Then $S$ is Koszul  and isomorphic to $TV/(\sfw)$ for some rank-two tensor $\sfw \in V^{\otimes 2}$. 
By \cite[Lem. 4.2]{MU}, for example, there exists $\s\in \Aut(S)$ such that $\hdet (\s)\neq \det(\s)$ if and only if $S\cong k\<x, y\>/(xy+yx)$ or,
equivalently, if and only if $\sfw \in\Sym^2V$. 

We will now show that something similar happens for the 3-dimensional case.

\subsection{}
Let $S=TV/(R)$ be a noetherian 2-Koszul AS-regular algebra of dimension 3.  Then $\dim V=3$.  

The symmetric group,  $\frak S_{3}$, acts on $V^{\otimes 3}$ by 
$$\theta(v_1\otimes v_2 \otimes v_{3}):=v_{\theta(1)}\otimes v_{\theta(2)} \otimes v_{\theta(3)},
\qquad \theta \in \fS_3.
$$  
Let $s, a:V^{\otimes 3}\to V^{\otimes 3}$ be the linear maps
\begin{align*}
& s(\sfw):=\frac{1}{|\frak S_{3}|}\sum _{\theta \in \frak S_{3}}\theta (\sfw) \\
& a(\sfw):=\frac{1}{|\frak S_{3}|}\sum _{\theta \in \frak S_{3}}(\operatorname{sgn}\theta) \theta(\sfw),
\end{align*}
and define the following subspaces of $V^{\otimes 3}$: 
\begin{align*}
& \operatorname{Sym}^{3}V:=\{\sfw \in V^{\otimes 3}\; | \; \theta(\sfw )=\sfw\;  \textnormal { for all } \theta\in \frak S_{3}\}, \\
& \operatorname{Alt}^{3}V:=\{\sfw \in V^{\otimes 3}\; | \; \theta(\sfw )=(\operatorname {sgn} \theta)\sfw\;  \textnormal { for all } \theta\in \frak S_{3}\}.
\end{align*}   
Fix a basis $\sfw_0$ for $\Alt^3(V)=k\sfw _0$ and let
$$
\mu :V^{\otimes 3}\to k
$$
be the unique map such that $c(\sfw)=s(\sfw)+\mu (\sfw)\sfw_0$. The map $\mu$ plays a key role in \cite {MS}.

\begin{theorem} \label{thm.hdd} 
Let $S=TV/(R)$ be a noetherian 2-Koszul AS-regular algebra of dimension 3.  
If  $\hdet (\s)\neq \det (\s)$ for some $\s\in \Aut(S)$,  then $c(\sfw_S)\in \Sym^3V$.  
\end{theorem} 

\begin{proof} 
Let $\sfw=\sfw_S$.
Although $\mu$ depends on the choice of $\sfw_0$, it is an easy exercise to show that 
(or, essentially by the definition of the determinant)
$\mu (\s^{\otimes 3}(\sfv))=(\det \s)\mu (\sfv)$ for all $\s\in \GL(V)$ and all $\sfv \in V^{\otimes 3}$.

By Theorem \ref{thm.msss}, if $\s \in \Aut(S)$, then $\s^{\otimes 3}(\sfw)=\hdet(\s)\sfw$ so 
$$\det (\s)\mu (\sfw)=\mu (\s^{\otimes 3}(\sfw))=\mu (\hdet (\s)\sfw)=\hdet (\s)\mu (\sfw).$$
If $\hdet (\s)\neq \det (\s)$, then $\mu (\sfw)=0$, so $c(\sfw)=\mu(\sfw)\sfw_0+s(\sfw)=s(\sfw)\in \Sym^3V$. 
\end{proof} 

For the rest of the paper, we write $\cD(\sfw):=\cD(\sfw, 1)$ and $J(\sfw):=J(\sfw, 1)$.

\begin{theorem} \label{thm.cwhd}
Let $V$ be a 3-dimensional vector space and $0\neq \sfw\in V^{\otimes 3}$ such that $J(\sfw)$ is Calabi-Yau. Then the following are equivalent: 
\begin{enumerate}
\item{} $\hdet (\s)\neq \det (\s)$ for some $\s\in \Aut(J(\sfw))$; 
\item{} $c(\sfw)\in \Sym^3V$;
\item{} $R\subseteq  \Sym^2V$. 
\end{enumerate}
\end{theorem}
 
\begin{proof} Let $S=J(\sfw)$.  By Corollary \ref{cor.cy2},  $\sfw_S=c(\sfw)$. 

 (1) $\Rightarrow$ (2) 
 By Theorem \ref{thm.hdd}, $c(\sfw)=c(c(\sfw))=c(\sfw_S)\in \Sym^3V$.

 (2) $\Rightarrow$ (1)
Since $c(\sfw)=s(\sfw)+a(\sfw)\in \Sym^3V\oplus \Alt^3V$, $c(\sfw)\in \Sym^3V$ if and only if $c(\sfw)=s(\sfw)$.  So, if $c(\sfw)\in \Sym^3V$ 
 there is a basis $x, y, z$ for $V$ such that 
$$c(\sfw)=xyz+yzx+zxy+xzy+yxz+zyx+\a x^3+\b y^3+\c z^3$$ 
and  $(\a, \b, \c) \in \{(1, 0, 0), (1, 1, 0), (\a, \a, \a) \; | \; \a \in k\}$  
by   \cite[\S1.8.4]{MS}.  
Let
$$
\s \; = \; 
\begin{cases}
\begin{pmatrix} 1 & 0 & 0 \\ 0 & 0 & 1 \\ 0 & 1 & 0 \end{pmatrix} & \text{if $(\a, \b, \c)$ is $(1, 0, 0)$ or $(\a, \a, \a)$}
\\
\begin{pmatrix} 0 & 1 & 0 \\ 1 & 0 & 0 \\ 0 & 0 & 1 \end{pmatrix} & \text{if $(\a, \b, \c)=(1, 1, 0)$.}
\end{cases}
$$
In all cases $\s^{\otimes 3}(\sfw)=\sfw$ so, by Theorem \ref{thm.ws},  $\s$ extends to an automorphism  of $J(\sfw)$ and, by Theorem \ref{thm.msss},
 $\hdet (\s)=1\neq -1=\det (\s)$.

 (2) $\Leftrightarrow $(3) 
This is the content of \cite [\S3.1]{MS}.  
\end{proof} 
 
\begin{remark} 
Let $V$ be a 3-dimensional vector space and $0\neq \sfw\in V^{\otimes 3}$.
By \cite [Thm. 3.2]{MS}, $c(\sfw)\not \in \Sym^3V$ if and only if $J(\sfw)$ is a deformation quantization of 
the polynomial ring $k[x, y, z]$.  
By Theorem \ref{thm.cwhd},  if $J(\sfw)$ is Calabi-Yau and a deformation quantization of $k[x, y, z]$, then 
$\hdet (\s)=\det (\s)$ for every $\s\in \Aut J(\sfw)$.  
\end{remark} 

The point of our final example is to show that the Zhang twists of a single noetherian Koszul Calabi-Yau algebra of dimension 3  
 can behave very differently from one another.  

\begin{example} 
Let $V=kx\oplus ky \oplus kz$ and $\sfw=xyz+yzx+zxy+xzy+yxz+zyx\in \Sym^3V$. Then   
$$
S\;=\;  \cD(\sfw) \; = \; J(\sfw) \; = \; \frac{k\<x, y, z\>}{(yz+zy, zx+xz, xy+yx)}\,.
$$
is a noetherian Koszul Calabi-Yau algebra of dimension 3. 
The linear maps $\s_i:V\to V$ defined by
$$
\s_1 \begin{pmatrix} x\\ y \\ z\end{pmatrix}  =\begin{pmatrix} 1 & 0 & 0 \\ 0 & 0 & 1 \\ 0 & 1 & 0 \end{pmatrix}\begin{pmatrix} x\\ y \\ z\end{pmatrix} 
\qquad \hbox{and} \qquad
\s_2 \begin{pmatrix} x\\ y \\ z\end{pmatrix}  =\begin{pmatrix} 0 & 1 & 0 \\ 0 & 0 & 1 \\ 1 & 0 & 0 \end{pmatrix}\begin{pmatrix} x\\ y \\ z\end{pmatrix} 
$$
satisfy $\s_i^{\otimes 3}(\sfw)=\sfw$ so $\s_1,\s_2  \in \Aut(S)$ and, by Theorem \ref{thm.msss}, 
$\hdet (\s_i)=1$.

Since $\s_2^3=\id$ and $\s_1^3 \ne \id$, it follows from Corollary \ref{cor.ocy} that $S^{\s_2}$ is Calabi-Yau but $S^{\s_1}$ is not.

A calculation shows that
\begin{align*}
\sfw^{\s_1} & \; = \; xz^2+y^2x+zxy+xy^2+yxz+z^2x \quad \hbox{and}
\\
\sfw^{\s_2} & \; = \; z^3+x^3+y^3+zxy+xyz+yzx
\end{align*}
so $c(\sfw^{\s_1})\in \Sym^3V$ but $c(\sfw^{\s_2})\notin \Sym^3V$. 

Since $c(\sfw^{\s_2})\notin \Sym^3V$, Theorem \ref{thm.hdd} says that $\hdet(\t)=\det( \t)$  for 
all $\t\in \Aut(S^{\s_2})$.  On the other hand, by Theorem \ref{thm.msn}, $\s_1 \in \Aut(S^{\s_1})$  and 
$\hdet_{S^{\s_1}}(\s_1)=\hdet_S(\s_1)=1 \neq -1 = \det (\s_1)$.
 
By Theorem \ref{thm.msn}(2),
 $$
S^{\s_1} \; \cong \; \cD(\sfw^{\s_1})\;  \cong \; \frac{k\<x, y, z\>}{(z^2+y^2, yx+xz, xy+zx)}
$$
and 
$$
S^{\s_2} \; \cong \; \cD(\sfw^{\s_2})\;  \cong \; 
\frac{k\<x, y, z\>}{(yz+x^2, zx+y^2, xy+z^2)}.
$$
\end{example}

\end{document}